\documentclass[reqno]{amsart}

\usepackage{amssymb,amsmath,amsthm,enumerate,amsrefs,xcolor,calc}
\usepackage{a4wide}
\usepackage[brazil,english]{babel}
\usepackage[utf8]{inputenc}

\newtheorem{teo}{Theorem}[section]
\newtheorem{prop}[teo]{Proposition}
\newtheorem{lema}[teo]{Lemma}
\newtheorem{coro}[teo]{Corollary}
	
	\theoremstyle{definition}
	
	\newtheorem{dfn}[teo]{Definition}
	
	\newtheorem{obs}[teo]{Remark}

	\newtheorem{ex}[teo]{Example}

\newcommand{\nd}{\noindent}
\newcommand{\vu}{\vspace{.1cm}}
\newcommand{\vd}{\vspace{.2cm}}

\newcommand{\Apd}{\leftharpoonup}
\newcommand{\Ape}{\rightharpoonup}
\newcommand\Ve\varepsilon

\renewcommand{\H}{{\mathcal{H}}}

\newcommand{\U}{{^{+1}}}
\newcommand{\Z}{{^{+0}}}
\newcommand{\Up}{{{}^{+\bar 1}}}
\newcommand{\Zp}{{{}^{+\bar 0}}}
\newcommand{\Um}{{^{-1}}}
\newcommand{\Zm}{{^{-0}}}
\newcommand{\Upm}{{{}^{-\bar 1}}}
\newcommand{\Zpm}{{{}^{-\bar 0}}}
\newcommand{\Ref}[1]{{\stackrel{(\ref{#1})}{=}}}
\newcommand{\Refie}[2]{{\stackrel{\ref{#1}(\ref{#2})}{=}}}
\newcommand{\rhop}{{\bar{\rho}}}
\newcommand{\ldap}{{\bar{\lambda}}}
\newcommand{\Sym}{{\stackrel{\mathrm{sym}}{=}}}

\newcommand{\Not}[2][0]{
	\setcounter{enumi}{#1}
	\renewcommand{\theenumi}{#2\arabic{enumi}}
	\renewcommand{\labelenumi}{(\theenumi)}
	\setlength{\itemindent}{\widthof{#2}}
	\setlength{\itemsep}{4pt}
}

\begin{document}\frenchspacing
	\title[Partial bi(co)module algebra]{Partial bi(co)module algebras, globalizations and partial (L,R)-smash products }

	\author[Castro]{Felipe Castro}
	\address{Instituto de Matem\'atica, Universidade Federal do Rio Grande do Sul,
		91509-900, Porto Alegre, RS, Brazil} \email{f.castro@ufrgs.br}
	\author[Paques]{Antonio Paques}
	\address{Instituto de Matem\'atica, Universidade Federal do Rio Grande do Sul,
		91509-900, Porto Alegre, RS, Brazil} \email{paques@mat.ufrgs.br}
	\author[Quadros]{Glauber Quadros}
	\address{Instituto de Matem\'atica, Universidade Federal do Rio Grande do Sul,
		91509-900, Porto Alegre, RS, Brazil} \email{glauber.quadros@ufrgs.br}
	\author[Sant'Ana]{Alveri Sant'Ana}
	\address{Instituto de Matem\'atica, Universidade Federal do Rio Grande do Sul,
		91509-900, Porto Alegre, RS, Brazil} \email{alveri@mat.ufrgs.br}

\thanks{{The first and the third authors were partially supported by CNPq, Brazil}.
}

	\begin{abstract}
		In this paper we introduce the notions of partial bimodule algebra and partial bicomodule algebra. We also deal with the existence of
globalizations for these structures, generalizing related results appeared in \cite{AB, AB2}.
As an application we construct the partial $(L,R)$-smash product, extending the corresponding global notion appeared in \cite{PO1} to the context of partial Hopf actions.
	\end{abstract}
	
	\maketitle
	
	\tableofcontents
\section{Introduction}${}$

\vd

Partial action of groups on algebras was introduced in the literature by R. Exel in \cite{E}. His main purpose in that paper was to
develop a method that allowed to describe the structure of $C^*$-algebras under actions of the circle group. The first approach of
partial group actions on algebras, in a purely algebraic context, appears later in a paper by M. Dokuchaev and R. Exel \cite{DE}.

\vu

Partial group actions can be easily obtained by restriction from the global ones, and this fact stimulated the interest on knowing under
what conditions (if any) a given partial group action is of this type. In the topological context this question was dealt with by F. Abadie
in \cite{abadie}. The algebraic version of a globalization (or enveloping action) of a partial group action, as well as, the study about its existence,
was also considered by Dokuchaev and Exel in \cite{DE}. A nice approach on the relevance of the relationship between partial and global
group actions, in several branches of mathematics, can be seen in \cite{D}.

\vu

Actions and coactions on algebras are an important part of the theory of Hopf algebras. Hence, as a natural task, S. Caenepeel and
K. Janssen \cite{CJ} extended the  notion of partial group action to the setting of Hopf algebras and developed a theory of partial
(co)actions of Hopf algebras. Based on the Caenepeel-Janssen's work, M. Alves and E. Batista in \cite{AB, AB1} showed that every
partial action of a Hopf algebra has a globalization.

\vu

The notion of $(L,R)$-smash product, in the context of global actions, was introduced and studied in a series of papers \cite{BBM1,BBM2,BGGS,BS},
in a first approach for cocommutative Hopf algebra, with motivation and examples coming from the theory of deformation quantization.
In \cite{PO1}, this notion was extended to the case of arbitrary bialgebras. Actually, in \cite{PO1} the authors introduce a more general
construction: they define an $(L,R)$-smash product $A\natural\bar{A}$, where $A$ is a bimodule algebra and $\bar{A}$ is a bicomodule algebra,
which motivated the main purposes of this paper, namely: to extend these notions of bi(co)module algebras to the context of partial
Hopf (co)actions and, as an application, to construct new associative algebras coming from partial (co)actions, namely,
the corresponding partial $(L,R)$-smash products, as well as, to deal with the existence of globalizations for these structures,
generalizing, in particular, the results in \cite{AB, AB2} concerning to this subject.

\vu

In the next section we recall the basic notions of partial Hopf (co)module algebras and introduce the notions of partial Hopf
bi(co)module algebras, illustrating them by several examples. In the Section 3 we deal with the globalization of these
structures, as well as, with the correspondence between partial $H$-bicomodule algebras and partial $H^0$-bimodule algebras,
showing, in particular, that  the globalizations of a partial $H^0$-bimodule algebra and of a partial $H$-bicomodule algebra
are isomorphic as $H^0$-bimodules, whenever $H^0$ separates points. In the last section we construct, as an application, the
partial $(L,R)$-smash product of a partial $H$-bimodule algebra with a partial $H$-bicomodule algebra.

\vu

Throughout, $H$ denotes a Hopf algebra over a field $\Bbbk$, with antipode $S$ and coalgebra structure given by the comultiplication
$\Delta:H\to H\otimes H, h\mapsto h_1\otimes h_2$, (Sweedler's notation with summation understood) and the counit $\varepsilon:H\to K$.
We will denote by $H^0$ its finite dual. Algebras are always considered over $\Bbbk$ and are associative but not necessarily unital.
By linear maps we mean $\Bbbk$-linear maps. Unadorned $\otimes$ means $\otimes_\Bbbk$.
	
\vd

\section{Partial Hopf bi(co)module algebras}

\vu

We start this section recalling the basic and known definitions of (global/partial) Hopf (co)module algebra structures,
before to introduce the notion of a partial bi(co)module algebra.

\vu

\subsection{Partial Hopf actions}${}$

\vd

\begin{dfn}\label{lma}
	An algebra $A$ is a \emph{left $H$-module algebra} if there exists a linear map $\triangleright:H\otimes A\to A$ denoted by $\triangleright(h\otimes a)=h\triangleright a$ such that:
	\begin{enumerate} \Not{LMA}
		\item $1_H\triangleright a=a,$
		\item $h\triangleright (a b)=(h_1\triangleright a)(h_2\triangleright b),$
		\item $h\triangleright (g\triangleright a)=h g\triangleright a$,
	\end{enumerate}
\end{dfn}
\noindent for all $a,b \in A$ and $g,h \in H$.

\vd

In a similar way, one defines a right $H$-module algebra. Recalling from \cite{CJ} (see also \cite{AB}), we have the definition of a partial Hopf  module algebra.

\begin{dfn}\label{lpma}
	An algebra $A$ is a \emph{left partial $H$-module algebra} if there exists a linear map $\Ape:H\otimes A \to A$, denoted by $\Ape(h\otimes a)=h\Ape a$, such that:
	\begin{enumerate} \Not{LPMA}
		\item $1_H\Ape a=a$,\label{lpma-1}
		\item $h\Ape[a(g\Ape b)]=(h_1\Ape a)(h_2g\Ape b)$,\label{lpma-2}
	\end{enumerate}
	\nd for all $a,b \in A$ and $g,h \in H$.
	
	\vd
	
	\nd We say that $A$ is \emph{symmetric} if the following additional condition holds:
	
	\vu
	\begin{enumerate} \Not[2]{LPMA}
		\item $h\Ape[(g\Ape b)a]=(h_1g\Ape b)(h_2\Ape a)$\label{lpma-3},
	\end{enumerate}
	\nd for all $a,b \in A$ and $g,h \in H$.
	
	\vd
	
	\nd The map $\Ape$ is called a \emph{left partial action of $H$ on $A$}, and we will also denote a left partial $H$-module algebra by the pair $(\Ape, A)$.
\end{dfn}

\vu

A (symmetric) right partial module algebra is defined in a similar way:

\vu

\begin{dfn}\label{rpma}
	An algebra $A$ is a \emph{right partial $H$-module algebra} if there exists a linear map $\Apd:A\otimes H \to A$, denoted by $\Apd(a\otimes h)=a\Apd h$ such that:
	
	\vu
	
	\begin{enumerate}\Not{RPMA}
		\item $a\Apd 1_H=a$;\label{rpma-1}
		\item $[(b\Apd g)a]\Apd h=(b\Apd gh_1)(a\Apd h_2)$,\label{rpma-2}
	\end{enumerate}
	\nd for all $a,b \in A$ and $g,h \in H$.
	
	\vu
	
	\nd And, it is \emph{symmetric} if the following additional condition holds
	
	\vu
	
	\begin{enumerate}\Not[2]{RPMA}
		\item $[a(b\Apd g)]\Apd h=(a\Apd h_1)(b\Apd gh_2)$,\label{rpma-3}
	\end{enumerate}
	
	\nd for all $a,b \in A$ and $g,h \in H$.
	
	\vd
	
	\nd The map $\Apd$ is called a \emph{right partial action of $H$ on $A$}, and we will also denote a partial right $H$-module algebra by the pair $(A, \Apd)$.
\end{dfn}

\begin{obs}\label{defcaen}
Definition \ref{lpma} was introduced by Caenepeel and Janssen in \cite{CJ}. It is a little bit different from the one given by Alves and Batista in \cite{AB}, because for the first authors, the algebra $A$ does not need to be unital.  
However when the algebra is unital the definitions are equivalent. Indeed,
if $A$ is a unital algebra, with identity element $1_A$, the conditions \ref{lpma-1} and \ref{lpma-2} of Definition \ref{lpma} are equivalent to the following ones:
	
	\vu
	
	\begin{enumerate}\Not{LPMA2.}
	    \item $1_H\Ape a=a$,
		\item $h\Ape(ab)=(h_1\Ape a)(h_2\Ape b)$,\label{lpma-2.1}	
		\item $h\Ape(g\Ape b)=(h_1\Ape 1_A)(h_2g\Ape b)$,\label{lpma-2.2}
	\end{enumerate}
	\nd for all $a,b \in A$ and $g,h \in H$.
	
	\vd
     In particular,  $\Ape$ is global if and only if $h\Ape 1_A=\varepsilon(h)1_A$,
	for all $h\in H$.
   An analogous remark also holds for right partial $H$-module algebras.
\end{obs}

\vu

In the sequel, we will present some examples of partial module algebras. We will see the relation between the unital partial actions of a group $G$
on an algebra $A$ and the symmetric partial actions of $\Bbbk G$ on the same algebra. We will also present an easy method to construct partial
module algebra from global ones. Actually, every partial Hopf action on unital algebras is of this type.

\vu

\begin{ex}\label{apG<=>apskG}
	(\cite[Proposition 4.9]{CJ}) There is a bijective correspondence between unital partial actions of a finite group $G$ on a unital algebra $A$ and symmetric
left partial actions of ${\Bbbk}G$ on $A$.

\vu
	
	By a unital partial action $\alpha$ of $G$ on $A$, we mean a collection of unital ideals $D_g$ of $A$, with identity element $1_g$, $g\in G$, and
algebra isomorphisms $\alpha_g:D_{g^{-1}}\to D_g$ such that
	\begin{enumerate}
		\item $D_1 = A$ and $\alpha_1$ is the identity automorphism $I_R$ of $A$,
		\item $\alpha_g(D_{g^{-1}}\cap D_h) = D_g\cap D_{gh}$,
		\item $\alpha_g \circ \alpha_h(r) = \alpha_{gh}(r)$, for every $r\in D_{h^{-1}} \cap D_{(gh)^{-1}}$.
	\end{enumerate}
	 Such a partial action determines a symmetric left partial action of $\Bbbk G$ on $A$ via the linear map $\Ape:{\Bbbk}G\otimes A \to A$ defined
    by $g\Ape a=\alpha_g(a1_{g^{-1}})$.

\vu
	
	Conversely, given any symmetric left partial action $\Ape$ of ${\Bbbk}G$ on $A$, the
	unital partial action of $G$ on $A$ is determined by the algebra isomorphisms  $\alpha_g:D_{g^{-1}}\to D_g$ defined by $\alpha_g(a1_{g^{-1}})=g\Ape a$,
where $1_g=g\Ape 1_A$ and $D_g=A1_g=1_gA$, for all $g\in G$.
\end{ex}

\vu

\begin{ex}\label{exemplo-eN}
	(see \cite{AB1}) This example, not coming from partial group actions, is obtained as a restriction of the global action of
 ${\Bbbk}G^\ast = Hom({\Bbbk}G,{\Bbbk})$ on ${\Bbbk}G$, given by $p_g \triangleright h = \delta_{g,h} h$, where $G$ is a finite group and $\{p_g\}_{g\in G}$
 is the dual basis for ${\Bbbk}G^\ast$.
	
	\vu
	
	Consider a normal subgroup $N$ of $G$, $N\neq\{1\}$, such that the characteristic of ${\Bbbk}$ does not divide the order of $N$. Take the central
idempotent $e_N\in {\Bbbk}G$ given by $$e_N=\frac{1}{|N|}\sum\limits_{n\in N}n$$ and consider $A = e_N{\Bbbk}G$. Clearly, $A$ is an ideal of ${\Bbbk}G$
and a unital algebra with $1_A = e_N$. The left partial action of ${\Bbbk}G^\ast$ over $A$ is given by
	$$
	p_g\Ape e_N h = 1_A(p_g\triangleright e_N h)=
		\begin{cases}
			\frac{1}{|N|}e_N h, &\text{if}\ g^{-1}h \in N\\
			0, &\text{otherwise}
		\end{cases}
	$$
	for all $h\in G$ and $g\in G$. Notice that $p_g\Ape 1_A \neq \varepsilon_{_{\Bbbk G^\ast}} (p_g)1_A$. Hence, this action is not global indeed.
\end{ex}

\vu

This above  construction is, in fact, a particular illustration of the general method to construct partial Hopf actions from global ones, which we see
in the next example (see \cite{AB, AB2}). We will consider left partial Hopf actions, for the right ones one proceeds in a similar way.

\vu

\begin{ex}\label{induced}
	Let $B$ be a left $H$-module algebra via the action $h \otimes b\mapsto h\triangleright b$, for all $b\in B$ and $h\in H$. If $A$ is a right ideal
generated by an idempotent $1_A$ such that it also is the identity element of $A$, then $A$ becomes a left partial $H$-module algebra via the partial action
	$$h\Ape a=1_A(h\triangleright a).$$
	
	Conversely, any partial action of $H$ on a unital algebra $A$ is of this type \cite[Theorem 1]{AB}.
\end{ex}

\vu

\begin{dfn}
	Let $H$ be a Hopf algebra and $A$ an algebra. We say that $A$ is a \emph{partial $H$-bimodule algebra} if
	there exist linear maps $\Ape:H\otimes A\to A$ and $\Apd:A\otimes H \to A$ such that $(\Ape, A)$
	is a left partial $H$-module algebra, $(A, \Apd)$ is a right partial $H$-module algebra and the following relation
	of compatibility is satisfied: $$h\Ape(a\Apd g) = (h\Ape a)\Apd g.$$

\vu

We will denote such a partial $H$-bimodule algebra by the triple $(\Ape, A, \Apd)$ or simply by $A$ if there is no risk of confusion. Notice that
any left or right partial $H$-module algebra can be seen as a partial $H$-bimodule algebra. Indeed, it suffices to take one of the partial actions as being the trivial action induced by $\varepsilon$.

\end{dfn}

We end this subsection with some illustrative  examples.

\begin{ex}
	Let $A$ be an algebra. Then, $\H = Hom_{\Bbbk}(H \otimes H, A)$ is an $H$-bimodule algebra with the actions given by
	$$
		\begin{array}{rcl}
			\triangleright: H \otimes \H & \to & \H\\
			g \otimes f & \mapsto & [g \triangleright f](h \otimes k) = f(h g \otimes k)
		\end{array}
	$$
	and
	$$
		\begin{array}{rcl}
			\triangleleft: \H \otimes H & \to & \H \\
			f \otimes g & \mapsto & [f \triangleleft g](h \otimes k) = f(h \otimes gk).
		\end{array}
	$$
\end{ex}

\begin{ex}
	Let $H_4 = \langle 1,g,x, xg\, |\, g^2=1, \ x^2 =0, \ xg=-gx \rangle$ be the Sweedler's algebra over the field ${\Bbbk}$. Then, ${\Bbbk}$ becomes a
partial $H_4$-bimodule algebra via the following linear maps
	$$
		\begin{array}{rcl@{\hspace*{1.5cm}}c@{\hspace*{1.5cm}}rcl}
			\Ape: H_4 \otimes {\Bbbk} & \to & {\Bbbk} &  & \Apd: {\Bbbk} \otimes H_4 & \to & {\Bbbk} \\
			1 \otimes 1_{\Bbbk} & \mapsto & 1_{\Bbbk} & & 1_{\Bbbk} \otimes 1 & \mapsto & 1_{\Bbbk} \\
			g \otimes 1_{\Bbbk} & \mapsto & 0 & \mathrm{and} & 1_{\Bbbk} \otimes g & \mapsto & 0 \\
			x \otimes 1_{\Bbbk} & \mapsto & r & & 1_{\Bbbk} \otimes x & \mapsto & s \\
			xg \otimes 1_{\Bbbk} & \mapsto & -r & & 1_{\Bbbk} \otimes xg & \mapsto & -s
		\end{array}
	$$
	for fixed $r,s \in {\Bbbk}$.
\end{ex}

\vu

Now we have an important example of partial bimodule algebra. In the same way we did in Example \ref{induced}, we can define an induced partial bimodule algebra from a global one as follows:

\vu

\begin{ex}\label{bi-induzida}
	Let $B$ be an $H$-bimodule algebra and $A$ a unital subalgebra of $B$ ($1_A$ not necessarily equal to $1_B$) such that for all $a,b\in A$ the following equality holds
	\begin{equation}
		(a \triangleleft h) (k \triangleright b) = (a \triangleleft h) 1_A (k \triangleright b)\label{cond-bi-induzida}
	\end{equation}
	and both of the sides  lie in $A$, where $\triangleleft$ and $\triangleright$ denotes the global actions of $H$ on $B$.
	
	\vu
	
	Then, $A$ becomes a partial $H$-bimodule algebra by the induced partial actions given by
	$$
		\begin{array}{rcl}
			\Ape\colon H\otimes A & \to & A\\
			h \otimes a & \mapsto & h\Ape a = 1_A (h \triangleright a)
		\end{array}
	$$
	
	$$
		\begin{array}{rcl}
			\Apd\colon A \otimes H & \to & A\\
			a \otimes h & \mapsto & a\Apd h = (a \triangleleft h) 1_{A} .
		\end{array}
	$$
	
	First of all, note that $\Ape$ and $\Apd$ are well defined using (\ref{cond-bi-induzida}). Clearly $A$ is a left and right partial $H$-module algebra, so it just remains to show the compatibility. In fact, for all $h,k \in H$ and $a \in A$,
	$$
	\begin{array}{rcl}
		h\Ape (a\Apd k)
		& = & 1_A (h\triangleright(a\Apd k))\\
		& = & 1_A (h\triangleright(a\Apd k))1_A\\
		& = & 1_A [h_1\triangleright(a\Apd k)(S(h_2)\triangleright 1_A)]\\
		& = & 1_A [h_1\triangleright(a\triangleleft k)1_A(S(h_2)\triangleright 1_A)]\\
		& = & 1_A [h_1\triangleright(a\triangleleft k)(S(h_2)\triangleright 1_A)]\\
		& = & 1_A (h\triangleright(a\triangleleft k))1_A\\
		& = & 1_A ((h\triangleright a)\triangleleft k)1_A\\
		& = & [(1_A \triangleleft S(k_1))(h\triangleright a)\triangleleft k_2]1_A\\
		& = & [(1_A \triangleleft S(k_1))1_A(h\triangleright a)\triangleleft k_2]1_A\\
		& = & [(1_A \triangleleft S(k_1))(h\Ape a)\triangleleft k_2]1_A\\
		& = & 1_A ((h\Ape a)\triangleleft k)1_A\\
		& = & 1_A[(h\Ape a)\Apd k]\\
		& = & (h\Ape a)\Apd k.
	\end{array}
	$$
	
	Therefore, $A$ is a partial $H$-bimodule algebra.

\vu

We will see in Subsection 3.1 that every partial $H$-bimodule algebra is of the type described above.
\end{ex}

\vd

\subsection{Partial Hopf coactions}${}$

\vd

First, we remember the definition of right $H$-comodule algebra.

\begin{dfn}\label{rca}
	A unital algebra $A$ is said a \emph{right $H$-comodule algebra} via the linear map $\rho:A \to A \otimes H$ if the following conditions hold:
	\begin{enumerate}\Not{RCA}
		\item $(I_A \otimes \Ve) \rho (a) = a$,\label{rca-1}
		\item $\rho (a b) = \rho(a) \rho(b)$,\label{rca-2}
		\item $(\rho \otimes I_H) \rho (a) = (I_A \otimes \Delta) \rho(a)$,\label{rca-3}
	\end{enumerate}
	for all $a, b\in A$. We will denote $\rho(a) = a\Z \otimes a\U$ (summation understood).
The map $\rho$ is called a \emph{right partial coaction of $H$ on $A$}, and we will also denote a right $H$-comodule algebra by the pair $(A, \rho)$.
\end{dfn}
\nd Left $H$-comodule algebras can also be defined in a similar way.

\vu

It is usual in the corresponding classical theory to put the additional axiom $\rho(1_A) = 1_A \otimes 1_H$ in the definition of an $H$-comodule algebra. Actually, this
is necessary whenever $H$ is only a bialgebra. In the more restrict case of Hopf algebras, this is a consequence from the above axioms and the existence of the antipode,
as follows in the next proposition.

\begin{prop}
	Let $A$ be a right $H$-comodule algebra. Then $\rho(1_A) = 1_A \otimes 1_H$.
\end{prop}

\begin{proof}In fact,
	$$
	\begin{array}{rcl}
	1_A \otimes 1_H & = & 1^{+0} \Ve(1^{+1}) \otimes 1_H \\
	& = &  1^{+0}  \otimes \Ve(1^{+1}) 1_H \\
	& = & 1^{+0} \otimes 1^{+1}{_1} S(1^{+1}{_2}) \\
	& = & 1^{+0+0} \otimes 1^{+0+1} S(1^{+1}) \\
	& = & (1_A 1^{+0})^{+0} \otimes (1_A 1^{+0})^{+1} S(1^{+1}) \\
	& = & 1^{+0'} 1^{+0+0} \otimes 1^{+1'} 1^{+0+1} S(1^{+1}) \\
	& = & 1^{+0'} 1^{+0} \otimes 1^{+1'} 1^{+1}{_1} S(1^{+1}{_2}) \\
	& = & 1^{+0'} 1^{+0}  \otimes 1^{+1'} \Ve(1^{+1}) \\
	& = & 1^{+0'} 1^{+0}\Ve(1^{+1})  \otimes 1^{+1'} \\
	& = & 1^{+0'} \otimes 1^{+1'} \\
	& = & \rho(1_A).
	\end{array}
	$$
\end{proof}

\vu

As the notion of partial module algebras generalizes module algebras, the notion of a partial comodule algebra arises as a generalization of a comodule algebra (see \cite{CJ} and \cite{AB}).

\begin{dfn}\label{rpca}
	A pair $(A, \rho)$, with $A$ a unital algebra and $\rho:A \to A \otimes H$ a linear map, is a \emph{right partial $H$-comodule algebra} if
	\begin{enumerate}\Not{RPCA}
		\item $\rho(ab) = \rho(a) \rho(b)$,\label{rpca-1}
		\item $(I_A \otimes \varepsilon)\rho(a)  = a$,\label{rpca-2}
		\item $(\rho \otimes I_H)\rho (a) = (\rho(1_A) \otimes 1_H)[(I_A \otimes \Delta)\rho (a)]$,\label{rpca-3}
	\end{enumerate} for all $a,b\in A$.
The coaction is said to be \emph{symmetric} if  the following additional condition is also satisfied:
	\begin{enumerate}\Not[3]{RPCA}
		\item $(\rho \otimes I_H)\rho (a) = [(I_A \otimes \Delta)\rho (a)](\rho(1_A) \otimes 1_H)$.\label{rpca-4}
	\end{enumerate}
	
	\nd If there is no risk of confusion we simply write $A$ instead of $(A,\rho)$ to denote a right partial $H$-comodule algebra. The map $\rho$ is called a \emph{right partial coaction of $H$ on $A$}.
\end{dfn}

For the left partial coactions, the axioms are similar.

\vu

\begin{dfn}\label{lpca}
	A pair $(\lambda, A)$, with  $A$ a unital algebra and $\lambda:A \to H\otimes A$ a linear map,  is called a \emph{left partial $H$-comodule algebra} if
	\begin{enumerate}\Not{LPCA}
		\item $(\varepsilon \otimes I_A) \lambda(a) = a$;\label{lpca-1}
		\item $\lambda(ab) = \lambda(a)\lambda(b)$;\label{lpca-2}
		\item $(I_{A} \otimes \lambda)\lambda(a) = [(\Delta \otimes I_{A})\lambda(a)] (1_H \otimes \lambda(1_{A}))$.\label{lpca-3}
	\end{enumerate}
	
	\nd for all  $a,b \in A$. The coaction is said to be \emph{symmetric} if  the following additional condition is also satisfied:
	\begin{enumerate}\Not[3]{LPCA}
		\item $(I_{A} \otimes \lambda)\lambda(a) = (1_H \otimes \lambda(1_{A})) [(\Delta \otimes I_{A})\lambda(a)]$.\label{lpca-4}
	\end{enumerate}
	
	\nd Also here, if there is no risk of confusion we write simply $A$ instead of $(\lambda, A)$ to denote a left partial $H$-comodule algebra.
The map $\lambda$ is called a \emph{left partial coaction of $H$ on $A$}.
\end{dfn}

\vu

\begin{obs} (Sweedler's notation with summation understood)
	If $A$ is a right partial $H$-comodule algebra, we will use the notation $\rho(a) = a\Zp \otimes a\Up$. Analogously, if $A$ is a left partial $H$-comodule
algebra, we will use the notation $\lambda(u) = u\Zpm \otimes u\Upm$. Moreover, recall that $\Delta(h) = h_1 \otimes h_2$ for all $h\in H$.
\end{obs}

\vu

It is straightforward to check that a right partial $H$-comodule algebra $A$ is global if and only if
$\rho(1_A)=1_A\otimes 1_H$. The same assertion also applies to the left case.

\vd

\begin{ex}
	(see \cite[section 6]{Lomp}) If $H$ is finite dimensional and $\Ape:H\otimes A \to A$ is a left partial action of $H$ on a unital algebra $A$, then $\rho:A \to A\otimes H^\ast $
given by $$\rho(a)=\sum_{1\leq i\leq n}(h_i\Ape a)\otimes p_i$$ is a right partial coaction of $H^\ast$ on $A$, where $\{h_i, p_i|1\leq i\leq n\}$ is a dual basis for $H$.
\end{ex}

\vu

This example was considered later in the more general context of Hopf algebras dually paired with a non-degenerated
pairing (see \cite[Theorem 4]{AB}).

\vu

Analogously to the case of action, we can define an induced partial coaction from a global one. We will consider right partial Hopf coactions,
for the left ones one proceeds in a similar way. This construction is described in the next example as follow.

\vu

\begin{ex}
	Let $B$ be a right $H$-comodule algebra via the coaction $\rho:B\to B\otimes H$. If $A$ is a right ideal of B with identity element $1_A$, then the map
	$$\bar{\rho}(a) = (1_A\otimes 1_H)\rho(a),\,\ \text{for all}\,\, a\in A,$$
	defines an structure of right partial $H$-comodule algebra on $A$ \cite[Proposition 8]{AB}. Conversely, any partial coaction  of $H$ on a unital algebra $A$ is of this type \cite[Theorem 4]{AB2}.
\end{ex}

\vu

Analogously to the action case, now we can define a partial bicomodule algebra. This is a left and right partial comodule algebra that obey a certain compatibility rule.

\vu

\begin{dfn}
	A unital algebra $A$ is said to be a \emph{partial $H$-bicomodule algebra} if there exist linear maps $\rho:A\to A\otimes H$ and $\lambda:A\to H\otimes A$
	such that $(A,\rho)$ is a right partial $H$-comodule algebra, $(\lambda, A)$ is a left partial $H$-comodule algebra and the following
	relation of compatibility is satisfied: $$(I_H \otimes \rho) \lambda = (\lambda \otimes I_H) \rho.$$

We will denote such a partial $H$-bicomodule algebra $A$ by the triple $(\lambda, A, \rho)$ or simply by $A$ if there is no risk of confusion. Notice that any left (resp., right) partial $H$-comodule algebra $A$ can be seen as partial $H$-bicomodule algebra. Indeed, it suffices to take one of the coactions as
being the trivial coaction given by the linear map induced by $a\mapsto 1_H\otimes a$ (resp., $a\mapsto a\otimes 1_H$), for all $a\in A$.

\end{dfn}

\vu

\begin{ex}\label{bi-co-induzida}
	Let $B$ be an $H$-bicomodule algebra and $A$ a unital subalgebra of $B$ ($1_A$ not necessary equal to $1_B$) such that the following equality holds
	\begin{equation}
		(\lambda(a) \otimes 1_H)(1_H \otimes \rho(b)) = (\lambda(a) \otimes 1_H)(1_H \otimes 1_A \otimes 1_H)(1_H \otimes \rho(b))\label{cond-bi-co-induzida}
	\end{equation}
	and both of the sides lie in $H \otimes A \otimes H$, for all  $a,b \in A$.

\vu	
	
	Then, $A$ becomes a partial $H$-bicomodule algebra via the coactions given by
	$$
	\begin{array}{rclcrcl}
	\bar{\rho}: A & \to & A \otimes H & \ \ \ \mathrm{and} \ \ \ & \bar{\lambda}: A & \to & H \otimes A \\
	a & \mapsto & (1_A \otimes 1_H) \rho(a)  & & a & \mapsto & \lambda(a)(1_H \otimes 1_A)
	\end{array}
	$$
	
	Clearly, applying $\Ve_H \otimes I_A \otimes I_H$ and $I_H \otimes I_A \otimes \Ve_H$ in  (\ref{cond-bi-co-induzida}), one sees that the maps $\bar\rho$ and
$\bar\lambda$ are well defined and $A$ is a left and right partial $H$-comodule algebra.
	
	It remains to show that the partial coactions are compatibles. Let $a \in A$. Then,
$$
	\begin{array}{rcl}
		(I \otimes \rhop) \ldap(a)
		& = & a\Um \otimes 1_A (a\Zm 1_A)\Z \otimes (a\Zm 1_A)\U \\
		& = & a\Um \otimes 1_A (a\Zm 1_A)\Z 1_A \otimes (a\Zm 1_A)\U \\
		& = & a\Um \otimes 1_A (a\Zm 1_A)\Z 1_A\Z \otimes (a\Zm 1_A)\U\Ve(1_A \U) \\
		& = & a\Um \otimes 1_A (a\Zm 1_A)\Z 1_A\Z \otimes (a\Zm 1_A)\U 1_A\U_1 S(1_A\U_2) \\
		& = & a\Um \otimes 1_A (a\Zm 1_A)\Z 1_A\Z\Z \otimes (a\Zm 1_A)\U 1_A\Z\U S(1_A\U) \\
		& = & a\Um \otimes 1_A (a\Zm 1_A 1_A\Z)\Z  \otimes (a\Zm 1_A 1_A\Z)\U S(1_A\U) \\
		& = & a\Um \otimes 1_A (a\Zm 1_A\Z)\Z  \otimes (a\Zm 1_A\Z)\U S(1_A\U) \\
		& = & a\Um \otimes 1_A a\Zm\Z 1_A\Z\Z \otimes a\Zm\U 1_A\Z\U S(1_A\U) \\
		& = & a\Um \otimes 1_A a\Zm\Z 1_A\Z \otimes a\Zm\U 1_A\U_1 S(1_A\U_2) \\
		& = & a\Um \otimes 1_A a\Zm\Z 1_A\Z \otimes a\Zm\U \Ve(1_A\U) \\
		& = & a\Um \otimes 1_A a\Zm\Z 1_A \otimes a\Zm\U
	\end{array}
$$
and a similar computation also shows that
$$(\ldap \otimes I)\rhop(a) = a\Z \Um \otimes 1_A a\Z\Zm 1_A \otimes a\Um.$$

Now, the result follows from the compatibility of $\rho$ and $\lambda$.

\vu

We will see in Subsection 3.2 that every partial $H$-bicomodule algebra is of the type described above.
\end{ex}

\vd

\section{Globalization for partial $H$-bi(co)module algebras}

\subsection{Globalization for partial $H$-bimodule algebras}${}$

\vd

Globalizations were first considered in \cite{abadie} for partial group actions on topological spaces, and later
in \cite{DE} for partial group actions on algebras and in \cite{AB} for partial Hopf (co)actions on algebras.

\vu

In this subsection we will introduce the notion of globalization of a partial Hopf bimodule structure, as well as,  we will ensure
its existence by constructing one, called standard, satisfying a certain condition of minimality. We will also prove that globalizations
satisfying this same condition of minimality are all isomorphic.

\vu

\begin{dfn}\label{biglobalizacao}
	Let $(\Ape, A, \Apd)$ be a unital partial $H$-bimodule algebra. A \emph{globalization} of $A$ is a pair $((\triangleright, B, \triangleleft),  \theta)$, where
$(\triangleright, B, \triangleleft)$ is an $H$-bimodule algebra
(not necessarily unital) and  $\theta:A\to B$ is an algebra monomorphism, such that the following conditions hold:
	\begin{enumerate}
		\item $(\theta(a)\triangleleft h)(g\triangleright\theta(b)) = \theta[(a\Apd h)(g\Ape b)]$,\label{glob-vesgo}\,\, $\text{for all}\,\ h,g\in H, a,b\in A$,
		\item $\theta$ is admissible, i.e., $B=H\triangleright\theta(A)\triangleleft H$.\label{glob-HAH}
	\end{enumerate}
\end{dfn}

\vd

The following proposition illustrates that, in fact, the above definition generalizes the notion of globalization of unital left (or, similarly, right) partial module algebras
in the sense of \cite{AB}.

\vu

\begin{prop}\label{motivacao}
	Let $(\Ape,A)$ be a unital left partial $H$-module algebra. Let $(\triangleright,B)$ be a left $H$-module algebra
such that there exists an algebra monomorphism $\theta:A \to B$  and $B = H \triangleright \theta(A)$. Consider on $A$ and $B$ the right trivial action given by $\varepsilon$.
Then the following statements are equivalent:
	\begin{itemize}
		\item[(1)] $(B, \theta)$ is a globalization of $A$.

\vu

		\item[(2)] $\theta(h \Ape a) = \theta(1_A)(h \triangleright \theta(a))$ for all $h \in H$ $a \in A$.

\vu

		\item[(3)] $(\theta(a) \triangleleft h)(k \triangleright \theta(b)) = \theta((a \Apd h)(k \Ape b))$ for all $h,k \in H$ $a,b \in A$.\label{motivacao-3}
	\end{itemize}
\end{prop}

\begin{proof}
	Clearly, (1) and (2) are equivalent (see \cite{AB}).

\vu
	
	We show now that (2) and (3) also are equivalent. Indeed, supposing $(2)$,
	$$
	\begin{array}{rcl}
	(\theta(a) \triangleleft h)(k \triangleright \theta(b))
	& = & \varepsilon(h) \theta(a)(k \triangleright \theta(b)) \\
	& = & \varepsilon(h) \theta(a)\theta(1_A)(k \triangleright \theta(b)) \\
	& = & \varepsilon(h) \theta(a)\theta(k \Ape b) \\
	& = & \theta(\varepsilon(h)a(k \Ape b))\\
	& = & \theta((a \Apd h)(k \Ape b)).
	\end{array}
	$$
	
	Conversely, supposing $(3)$,
	$$
	\begin{array}{rcl}
	\theta(1_A)(h \triangleright \theta(a)) & = & (\theta(1_A) \triangleleft 1_H )(h \triangleright \theta(a)) \\
	& = & \theta((1_A \Apd 1_H)(h \Ape a)) \\
	& = & \theta(1_A(h \Ape a))\\
	& = & \theta(h \Ape a).
	\end{array}
	$$
\end{proof}

\vu

\begin{coro}\label{cor1}
Let $(\Ape, A)$, $(\triangleright,B)$ and $\theta:A \to B$ be as in Proposition \ref{motivacao}. Consider on $A$ and $B$ the right trivial action given by $\varepsilon$.
Then, $(B, \theta)$ is a globalization of $A$ as a left partial $H$- module algebra (in the sense of \cite{AB}) if and only if $(B, \theta)$ is a globalization of $A$,
as a partial $H$-bimodule algebra.
\end{coro}

\vu

Before to prove this corollary we need some additional details about globalizations of partial bimodule algebras, which will be also useful in the sequel.

\vu

\begin{lema}\label{lemaco}
Let $(\Ape, A, \Apd)$ be a unital partial $H$-bimodule algebra and $((\triangleright, B, \triangleleft), \theta)$ a globalization of $A$. Then,
	\begin{itemize}
		\item[(i)] $(h\triangleright\theta(a)\triangleleft k)(h'\triangleright\theta(b)\triangleleft k')=
                   h_1\triangleright\theta[(a\Apd kS(k'_1))(S(h_2)h'\Ape b)]\triangleleft k'_2$,\label{lemaco-1}

\vu

		\item[(ii)] $\theta(1_A)(h\triangleright\theta(a))=\theta(h\Ape a)$,\label{lemaco-2}

\vu

		\item[(iii)]  $(\theta(a)\triangleleft h)\theta(1_A)=\theta(a\Apd h)$,\label{lemaco-3}

\vu

		\item[(iv)]  $\theta(1_A)(h\triangleright\theta(a)\triangleleft k)\theta(1_A)=\theta(h\Ape a\Apd k)$,\label{lemaco-4}
	\end{itemize} for all $a,b\in A$ and $h,h',k,k'\in H$.
\end{lema}

\begin{proof} The items (ii)-(iv) easily follows from the first item. Indeed, for (ii) put $h = k = k' = 1_H$ and $a = 1_A$ in (i),
for (iii) put $h = h' = k' = 1_H$ and $b = 1_A$ in (i), and for (iv),

$$
\begin{array}{rcl}
\theta(1_A)[h\triangleright\theta(a)\triangleleft k]\theta(1_A)
& = & \theta(1_A)[h\triangleright\theta(a)\triangleleft k][1_H \triangleright \theta(1_A) \triangleleft 1_H]\\
& \overset{(i)}{=} & \theta(1_A)(h_1\triangleright \theta((a \Apd k)(S(h_2) \Ape 1_A))\triangleleft 1_H)\\
& = & \theta(1_A)(h_1\triangleright \theta((a \Apd k) (S(h_2) \Ape 1_A)))\\
& \overset{(ii)}{=} & \theta(h_1 \Ape [(a \Apd k) (S(h_2) \Ape 1_A)])\\
& = & \theta((h_1 \Ape (a \Apd k)) (h_2 S(h_3) \Ape 1_A))\\
& = & \theta((h_1 \Ape (a \Apd k)) (\Ve(h_2)1_H \Ape 1_A))\\
& = & \theta((h_1\Ve(h_2) \Ape (a \Apd k)) (1_H \Ape 1_A))\\
& = & \theta((h \Ape a \Apd k) 1_A)\\
& = & \theta(h\Ape a\Apd k).
\end{array}
$$

\vu

Hence, it remains to prove (i). For this, we start pointing out the following two equations, which
are straightforward and will be useful in what follows,
\begin{equation}
y(x\Apd h) = [(y\Apd S(h_1))x]\Apd h_2
\end{equation} and
\begin{equation}
(h\Ape x)y = h_1\Ape [x(S(h_2)\Ape y)].
\end{equation}

\vu

Now, for all  $a, b\in A$, $h,h',k,k'\in H$ we have
	$$
	\begin{array}{rcl}
	(h\triangleright\theta(a)\triangleleft h')(k\triangleright\theta(b)\triangleleft k')
	&\overset{(4)}{=} & h_1\triangleright[(\theta(a)\triangleleft h')  (S(h_2)\triangleright(k\triangleright\theta(b)\triangleleft k'))]\\
	& = & h_1\triangleright[(\theta(a)\triangleleft h')(S(h_2)k\triangleright\theta(b)\triangleleft k')]\\
	&\overset{(3)}{=} & h_1\triangleright[(\theta(a)\triangleleft h' S(k'_1))(S(h_2)k\triangleright\theta(b))]\triangleleft k'_2\\
	& \Refie{biglobalizacao}{glob-vesgo} & h_1\triangleright\theta[(a\Apd h' S(k'_1))(S(h_2)k\Ape\theta(b))]\triangleleft k'_2.
	\end{array}
	$$
\end{proof}

\vu

\begin{proof} (of Corollary \ref{cor1})

\vu

First of all, notice that $H \triangleright \theta(A) \triangleleft H = H \triangleright \theta(A)\Ve(H)=H \triangleright \theta(A)$.

\vu

Assuming that $(B,\theta)$ is a globalization of $A$ as a left partial $H$-module algebra. it follows from Proposição \ref{motivacao} that
$(\theta(a) \triangleleft h)(k \triangleright \theta(b)) = \theta((a \Apd h)(k \Ape b))$. Hence, $(B,\theta)$ is a globalização of $A$ as a partial
$H$-bimodule algebra.

\vu

Conversely, if $(B,\theta)$ is a globalization of $A$ as a partial $H$-bimodule algebra, then $\theta(1_A)(h\triangleright\theta(a))=\theta(h\Ape a)$
by Lema \ref{lemaco} and, therefore, $(B,\theta)$ is a globalização of $A$ as a left partial $H$-module algebra by Proposição \ref{motivacao}.
\end{proof}

\vd

In the sequel we will proceed with the construction of a globalization for a unital partial $H$-bimodule algebra.

\vd

Let $(\Ape, A, \Apd)$ be a unital partial $H$-bimodule algebra and consider the convolutive algebra $\H = Hom_{\Bbbk}(H\otimes H, A)$.
It is immediate to check that the map
	$$
	\begin{array}{rcl}
	\varphi:A & \to & \mathcal{H}\\
	a & \mapsto & \varphi(a): k\otimes k'\mapsto k\Ape a\Apd k'
	\end{array}
	$$
is an algebra monomorphism. As well, it is straightforward that $\H$ has an structure of
$H$-bimodule given by $$(h\triangleright f)(k\otimes k') = f(kh\otimes k')$$
and $$(f\triangleleft h)(k\otimes k') = f(k\otimes hk').$$

\vd

The map $\varphi$ as above defined satisfies the condition $(\ref{glob-vesgo})$ of the definition of glo\-ba\-li\-za\-ti\-on, that is,
$$\varphi((b\Apd h')(h\Ape a)) = (\varphi(b)\triangleleft h')\ast(h\triangleright\varphi(a)).$$
Indeed,  for all $h,h',k,k'\in H$ and $a,b\in A$,
	
	$$
	\begin{array}{rcl}
	\varphi((b\Apd h')(h\Ape a))(k\otimes k') & = & k\Ape (b\Apd h')(h\Ape a)\Apd k'\\
	& = & (k_1\Ape b\Apd h'\Apd k'_1)(k_2\Ape h\Ape a\Apd k'_2)\\
	& = & (k_1\Ape b\Apd h'\Apd k'_1)(k_2\Ape 1_A)(k_3 h\Ape a\Apd k'_2)\\
	& = & (k_1\Ape b\Apd h'\Apd k'_1)(k_2 h\Ape a\Apd k'_2)\\
	& = & (k_1\Ape b\Apd h' k'_1)(1_A\Apd k'_2)(k_2\Ape h\Ape a\Apd k'_3)\\
	& = & (k_1\Ape b\Apd h' k'_1)(k_2 h\Ape a\Apd k'_2)\\
	& = &[\varphi(b)(k_1\otimes h'k'_1)][\varphi(a)(k_2h\otimes k'_2)]\\
	& = &[(\varphi(b)\triangleleft h')(k_1\otimes k'_1)][(h\triangleright\varphi(a))(k_2\otimes k'_2)]\\
	& = &[(\varphi(b)\triangleleft h')\ast(h\triangleright\varphi(a))](k\otimes k').
	\end{array}
	$$

\vd

Finally, it follows from Lemma \ref{lemaco} that $B = H\triangleright\varphi(A)\triangleleft H$ is an algebra. Therefore, $B$  is a globalization of $A$,
and the proof of the following theorem is accomplished.

\vu

\begin{teo}
	Let $(\Ape, A, \Apd)$ be a unital partial $H$-bimodule algebra. Then, the pair $(B, \varphi)$ is a globalization of $A$. It is called the standard globalization of $A$.\qed
\end{teo}

\vu

\begin{obs} With respect  to the standard globalization, we observe, in addition, that

\vu

		(1) if the right partial action on $A$ is symmetric,  $\varphi(A)$ is a right ideal of $B$. Indeed,	
		
		$$
		\begin{array}{rcl}
		\varphi(b(h\Ape a\Apd h'))(k\otimes k')
		& = & k\Ape b(h\Ape a\Apd h')\Apd k' \\
		& = & (k_1\Ape b\Apd k'_1)(k_2\Ape h\Ape a\Apd h'\Apd k'_2)\\
		& = & (k_1\Ape b\Apd k'_1)(k_2\Ape 1_A)(k_3 h\Ape a\Apd h'\Apd k'_2)\\
		& = & (k_1\Ape b\Apd k'_1)(k_2 h\Ape a\Apd h'\Apd k'_2)\\
		& \Sym & (k_1\Ape b\Apd k'_1)(1_A\Apd k'_2)(k_2 h\Ape a\Apd h' k'_3)\\
		& = & (k_1\Ape b\Apd k'_1)(k_2 h\Ape a\Apd h' k'_2)\\
		& = & \varphi(b)(k_1\otimes k'_1)[h\triangleright\varphi(a)\triangleleft h')(k_2\otimes k'_2)]\\
		& = & [\varphi(b)\ast(h\triangleright\varphi(a)\triangleleft h')](k\otimes k'),
		\end{array}
		$$ for all $a, b\in A$, $h, h', k, k'\in H$.

\vu

		(2) and, similarly, if the left partial action on $A$ is symmetric, $\varphi(A)$ is a left ideal of $B$.
\end{obs}

\vu

Below, we will see the relation between any globalization and the standard one. In the case of left partial actions, in \cite{AB} the authors have shown that there exists an algebra epimorphism from any globalization to the standard one. A analogous result also holds for partial bimodule algebras.

\begin{teo}\label{teo3}
	Let $(B', \theta)$ be a globalization of a unital partial $H$-bimodule algebra $A$. Then, there exists an algebra epimorphism from $(B',\theta)$ onto the standard globalization $(B,\varphi)$.
\end{teo}

\begin{proof}
	Since $B$ and $B'$ are globalizations for $A$ we have $B = H \triangleright \varphi(A) \triangleleft H$ and $B' = H \triangleright \theta(A) \triangleleft H$. Thus, we define
	$$
	\begin{array}{rcl}
	\Phi:  (B',\theta) & \to & (B,\varphi)\\
	h \triangleright \theta(a) \triangleleft k & \mapsto & h \triangleright \varphi(a) \triangleleft k
	\end{array}
	$$
	that is well-defined by the definition of $\varphi$ and by the injectivity of $\theta$. Indeed, suppose $x = \sum h_i \triangleright \theta(a_i) \triangleleft k_i = 0$. So, for all $h', k'$ in H, we have
	$$
	\begin{array}{rcl}
	0
	& = & \theta(1_A)( h' \triangleright \sum{h_i \triangleright \theta(a_i) \triangleleft k_i} \triangleleft k')\theta(1_A) \\
	& = & \theta(1_A)( \sum h' h_i \triangleright \theta(a_i) \triangleleft k_i k')\theta(1_A) \\
	& = & \theta(\sum h' h_i \Ape a_i \Apd k_i k').
	\end{array}
	$$
	
	Since $\theta$ is injective, we have $\sum h' h_i \Ape a_i \Apd k_i k' = 0$ for all $h', k'$ in $H$. So,
	$$
	\begin{array}{rcl}
	\Phi(x) (h' \otimes k')  & = & \sum{(h_i \triangleright \varphi(a_i) \triangleleft k_i)} (h' \otimes k') \\
	& = & \sum h' h_i \Ape a_i \Apd k_i k' \\
	& = & 0
	\end{array}
	$$
	and $\Phi$ is well defined. From Lemma \ref{lemaco}, it follows that $\Psi$ is an algebra epimorphism.
\end{proof}

Motivated by the above theorem we have the following definition for a minimal globalization.

\begin{dfn}
	Let $A$ be a partial bimodule algebra. A globalization $(B',\theta)$ is called \emph{minimal} if for any sub-bimodule $M$ of $B'$ in which $\theta(1_A) M \theta(1_A) = 0$ we have that $M = 0$.
\end{dfn}
A natural question arises from this definition: What relation is there between the notion of a minimal globalization and the bijectivity of the map $\Phi$ defined in Theorem \ref{teo3}? The answer to this question is given in the next remark.

\begin{obs}
	If $(B', \theta)$ is a minimal globalization for $A$, then $\Phi$ is an isomorphism.

\vu
	
	In fact, if $x = \sum{h_i \triangleright \theta(a_i) \triangleleft k_i}$ lies in kernel of $\Phi$ we have
	$$
	\begin{array}{rcl}
	0
	& = & (h \triangleright \Phi(x) \triangleleft k)(1_H \otimes 1_H)\\
	& = & (h \triangleright \sum{h_i \triangleright \varphi(a_i) \triangleleft k_i} \triangleleft k )(1_H \otimes 1_H)\\
	& = & \sum{h h_i \Ape a_i \Apd k_i k}
	\end{array}
	$$
	for all $h,k$ in H. Then, taking $M$ as the sub-bimodule generated by $\sum{h_i \triangleright \theta(a_i) \triangleleft k_i}$, we have
	$$
	\begin{array}{rcl}
	\theta(1_A)(\sum{h \triangleright h_i \triangleright \theta(a_i) \triangleleft k_i \triangleleft k})\theta(1_A)
	& = & \theta(1_A)(\sum{h h_i \triangleright \theta(a_i) \triangleleft k_i k})\theta(1_A) \\
	& \Ref{lemaco} & \theta(\sum{ h h_i \Ape a_i \Apd k_i k })\\
	& = & 0.
	\end{array}
	$$
	
	It follows from the minimality of $B'$ that $M=0$ and so $x = 0$. Therefore, $\Phi$ is an isomorphism.
\end{obs}

\vu

\subsection{Globalization for partial $H$-bicomodule algebras}${}$

\vd

We start this section discussing about the relation between partial bicoactions and partial biactions. This will be useful to construct  a globalization for partial bicomodule algebras.

\vu

Throughout this section we will denote by $\triangleright$ and $\triangleleft$ the (global) actions and by $\Ape$ and $\Apd$ the partial ones. When on doubt, we will use the global notation. We recall that the finite dual $H^0$ of a Hopf algebra $H$ is also a Hopf algebra, and $H^0 = H^\ast$ if $H$ is finite dimensional.

\vu

\begin{dfn}
	We say that $H^0$ \emph{separates points} if for all non-zero $h \in H$ there exists $f \in H^0$ such that $f(h) \not= 0$.
\end{dfn}

Let $B$ be an algebra and $\rho:B \to B\otimes H$ a linear map, denoted by $\rho(b)=b^{+0}\otimes b^{+1}$. We can define the induced linear map $\triangleright_\rho:H^\ast\otimes B\to B$ by $f\triangleright b= b^{+0}f(b^{+1})$. It is well known that $A$ is an $H$-comodule via $\rho$ if and only if it is an $H^\ast$-module via $\triangleright_\rho$. Moreover, the Jacobson's density theorem ensures that the same statement remains  true replacing $H^\ast$ by $H^0$, whenever $H^0$ separates points.

\begin{teo}
	\label{Hcmag=>H0mag}
	If $A$ is a right partial $H$-comodule algebra, then $A$ is a left partial $H^0$-module algebra via the corresponding induced action. \qed
\end{teo}

And, as a converse of the above theorem, we have the following result.

\begin{teo}
	Let $A$ be a unital  algebra and $\rho: A \to A \otimes H$ a linear map. Assume that $H^0$ separates points. If  $A$ is a left partial $H^0$-module algebra via the induced action, then $(A,\rho)$ is a right partial $H$-comodule algebra.
\end{teo}

\begin{proof}
	Suppose that $A$ is a left partial $H^0$-module algebra. For all $f,g \in H^0$ and $a,b \in A$
	\vd
	$$
	\begin{array}{lrcl}
	(\ref{rpca-1})   & (I_A \otimes f) (\rho (ab))
	& = &  (ab)^{+0} f((ab)^{+1}) \\
	& & = &  f \triangleright ab  \\
	& & = &   (f_1 \triangleright a)(f_2 \triangleright b)\\
	& & = &   a^{+0} f_1(a^{+1}) b^{+0} f_2(b^{+1}) \\
	& & = &   a^{+0} b^{+0} f(a^{+1} b^{+1})\\
	& & = &   (I_A \otimes f)(\rho (a) \rho (b) ) \\
    & & & \\
	(\ref{rpca-2})    & (I_A \otimes \varepsilon_H) \rho(a)
	& = &   a^{+0} \varepsilon(a^{+1}) \\
	& & = &   \varepsilon \triangleright a \\
	& & = &    1_{H^0} \triangleright a \\
	& & = &    a\\
	& & &  \\
	(\ref{rpca-3})    & (I_A \otimes f \otimes g)((\rho \otimes I_H) \rho (a))
	& = &  a^{+0+0} f(a^{+0+1}) g(a^{+1})  \\
	& & = &  (f \triangleright a^{+0}) g(a^{+1})  \\
	& & = &  f \triangleright (g \triangleright a)  \\
	& & = &  (f_1 \triangleright 1_A)(f_2 \ast g \triangleright a) \\
	& & = &  (f_1({1_A}^{+1}){1_A}^{+0})(f_2 \ast g (a^{+1})a^{+0})  \\
	& & = &  (f_1({1_A}^{+1}){1_A}^{+0})(f_2 (a^{+1}{}_1) g(a^{+1}{}_{2}) a^{+0})  \\
	& & = &  f_1({1_A}^{+1})f_2 (a^{+1}{_1})g(a^{+1}{_2}) {1_A}^{+0} a^{+0}  \\
	& & = &  f({1_A}^{+1} a^{+1}{_1})g(a^{+1}{_2}) {1_A}^{+0} a^{+0}  \\
	& & = &  (I_A \otimes f \otimes g) ({1_A}^{+0} a^{+0} \otimes {1_A}^{+1} a^{+1}{_1}  \otimes a^{+1}{_2} ) \\
	& & = &  (I_A \otimes f \otimes g)( (\rho(1_A) \otimes 1_H) (I_A \otimes \Delta) \rho (a)).
	\end{array}
	$$
	Since $H^0$ separates points, the required statement follows.
\end{proof}

\begin{obs}
	The same statements can be obtained replacing the left partial action by a right one and the right partial coaction by a left one. Moreover, these partial actions are global if and only if the partial coactions are global.
\end{obs}

\begin{prop}
	\label{Hbcma=>H0bma}
	Let $A$ be a unital algebra, $\rho:A \to A \otimes H$, $\lambda:A \to H \otimes A$ linear maps, and $\triangleright$ and $\triangleleft$ the respective induced maps. If $(\lambda, A, \rho)$ is a (partial) $H$-bicomodule algebra, then $(\triangleright, A, \triangleleft)$ is a (partial) $H^0$-bimodule algebra. The converse is also true whenever $H^0$ separates points.
\end{prop}

\begin{proof}It only remains to show the compatibility between the (partial) (co)actions. Let $a,b \in A$ and $f,g \in H^0$, so
	$$
	\begin{array}{rcl}
	f \triangleright (a \triangleleft g)
	& = & f \triangleright (g(a^{-1}) a^{-0}) \\
	& = & g(a^{-1}) a^{-0 +0} f (a^{-0 +1}) \\
	& = & (g\otimes I_A\otimes f)(a^{-1}\otimes a^{-0 +0} \otimes a^{-0 +1}) \\
	& = & (g\otimes I_A\otimes f)[(I_H\otimes\rho)\lambda(a)]
	\end{array}
	$$
	and
	$$
	\begin{array}{rcl}
	(f \triangleright a) \triangleleft g
	& = & (f (a^{+1}) a^{+0} ) \triangleleft g\\
	& = & (g\otimes I_A\otimes f)(a^{+0 -1}\otimes a^{+0 -0} \otimes a^{+1})\\
	& = & (g\otimes I_A\otimes f)[(\lambda\otimes I_H)\rho(a)].
	\end{array}
	$$
\end{proof}

The following proposition will be useful in the sequel.

\begin{prop}\label{vesgosequiv}
	Assume that $H^0$ separate points. Let $B$ be an $H$-bicomodule algebra and consider on $B$ the induced structure of
	$H^0$-bimodule algebra. If there exists a unital subalgebra $A$ of $B$, then the following statements are equivalent:

\vu
	
	\begin{itemize}
		\item[(i)] $(a \triangleleft f)(g \triangleright b) = (a \triangleleft f) 1_A (g \triangleright b)$ in $A$, $\text{for all}\ a,b \in A$ and $f,g \in H^0$.

\vu
		\item[(ii)] $(\lambda(a) \otimes 1_H)(1_H \otimes \rho(b)) = (\lambda(a) \otimes 1_H)(1_H \otimes 1_A \otimes 1_H)(1_H \otimes \rho(b))$ in $H \otimes A \otimes H$, $\text{for all}\ a,b \in A$.\label{vesgoequiv-2}
	\end{itemize}
\end{prop}

\begin{proof}
	First note that (ii) can be rewritten as $$a^{-1} \otimes a^{-0} b^{+0} \otimes b^{+1} =  a^{-1} \otimes a^{-0} 1_A b^{+0} \otimes b^{+1},$$ for all $a,b \in A$.
	
	Supposing (i), for all $f,g \in H^0$ we have
	$$
	\begin{array}{r@{~}c@{~}l}
	(f \otimes I_A \otimes g)((\lambda(a) \otimes 1_H)(1_H \otimes \rho(b)))
	& = & (f \otimes I_A \otimes g)(a^{-1} \otimes a^{-0} b^{+0} \otimes b^{+1})\\
	& = & f(a^{-1}) a^{-0} b^{+0} g(b^{+1}) \\
	& = & (a \triangleleft f)(g \triangleright b) \\
	& = & (a \triangleleft f)1_A(g \triangleright b) \\
	& = & f(a^{-1}) a^{-0} 1_A b^{+0} g(b^{+1}) \\
	& = & f(a^{-1}) \otimes a^{-0} 1_A b^{+0} \otimes g(b^{+1}) \\
	& = & (f \otimes I_A \otimes g)(a^{-1} \otimes a^{-0} 1_A b^{+0} \otimes b^{+1})\\
	& = & (f \otimes I_A \otimes g)((\lambda(a) \otimes 1_H)(1_H \otimes 1_A \otimes 1_H)(1_H \otimes \rho(b)))
	\end{array}
	$$
		which implies $$(\lambda(a) \otimes 1_H)(1_H \otimes \rho(b)) = (\lambda(a) \otimes 1_H)(1_H \otimes 1_A \otimes 1_H)(1_H \otimes \rho(b)),$$ and, clearly, it lies in $H \otimes A \otimes H$.

\vu

	Conversely, for all $a,b \in A$ and $f,g \in H^0$,
	
	$$
	\begin{array}{rl}
	(a \triangleleft f)(g \triangleright b) & = f(a^{-1})a^{-0}  b^{+0} g(b^{+1}) \\
	& = f(a^{-1})a^{-0} 1_A  b^{+0} g(b^{+1}) \\
	& = (a \triangleleft f)1_A(g \triangleright b)
	\end{array}
	$$
	and it also, clearly, lies in $A$.
	\end{proof}

Recalling from the Example \ref{bi-co-induzida}, if a bicomodule algebra $(\lambda, B, \rho)$ and a unital subalgebra $A$ ($1_A$ not equal to $1_B$)
are given such that  $(\lambda(a) \otimes 1_H)(1_H \otimes \rho(b)) = (\lambda(a) \otimes 1_H)(1_H \otimes 1_A \otimes 1_H)(1_H \otimes \rho(b))$
in $H \otimes A \otimes H$,  for all $a,b \in A$, then one can induce on $A$ an structure of a partial $H$-bicomodule algebra.

Is there a converse of this fact? That is, given any partial $H$-bicomodule algebra $A$ is there a way to obtain an $H$-bicomodule algebra $B$
such that the partial coactions on $A$ are induced by the corresponding coations on $B$?

To answer this question, similarly to the case
of partial bimodule algebras, we need first to introduce the notion of  globalization for partial bicomodule algebras.

\begin{dfn}
	A \emph{globalization} of a partial $H$-bicomodule algebra $(\ldap, A, \rhop)$, is a pair $((\lambda, B, \rho), \theta)$, where $(\lambda, B, \rho)$
is an $H$-bicomodule algebra (with $B$ not necessarily unital) and $\theta:A\to B$ is an algebra monomorphism  such that

\vu

	\begin{enumerate}
		\item $B$ is the $H$-bicomodule algebra  generated by $\theta(A)$, i.e., the smallest $H$-bicomodule algebra containing $\theta(A)$,

\vu

		\item $(\lambda( \theta(a)) \otimes 1_H)(1_H \otimes \rho(\theta(b))) = (I_H \otimes \theta \otimes I_H )((\ldap( a) \otimes 1_H)(1_H \otimes \rhop(b)))$	for all $a,b \in A$.
	\end{enumerate}
\end{dfn}

We will dedicate the rest of this subsection to prove that any partial $H$-bicomodule algebra has a globalization. We will proceed by steps.

\begin{lema}\label{H-subbicomodulogerado}
	Let $(\lambda, B, \rho) $ be an $H$-bicomodule algebra, $A\subseteq B$ a subalgebra and $C=H^\ast\triangleright A\triangleleft H^\ast$, where
$\triangleright$ and $\triangleleft$ are the corresponding actions of $H^*$ on $B$ induced by $\lambda$ and $\rho$ respectively. Then,
the algebra generated by $C$ is the smallest $H$-subbicomodule algebra of $B$ containing $A$.
\end{lema}

\begin{proof}
	Since $B$ is an $H$-bicomodule, then $B$ is a left and right $H^\ast$-module. Moreover, $B$ is an $H^\ast$-bimodule. In fact, let $f,g\in H^\ast$ and $b\in B$, so
	$$
	\begin{array}{rcl}
	(f\triangleright a)\triangleleft g
	& = & [a^{+0}f(a^{+1})]\triangleleft g\\
	& = & g(a^{+0-1})a^{+0-0}f(a^{+1})\\
	& = & g(a^{-1})a^{-0+0}f(a^{-0+1})\\
	& = & g(a^{-1})(f\triangleright a^{-0})\\
	& = & (f\triangleright g(a^{-1})a^{-0})\\
	& = & f\triangleright (a\triangleleft g).
	\end{array}
	$$
	
	Now we will show that $C$ is a left (and right) $H^\ast$-rational submodule  of $B$.

\vu

	Since $C=\sum\limits_{a\in A} H^\ast\triangleright a\triangleleft H^\ast$, we only need to show that $H^\ast\triangleright a\triangleleft H^\ast$ is rational, for each $a\in A$. Let $\{h_i\}$ be a basis of $H$ over ${\Bbbk}$. Then, there exist elements $a_{i,j}\in A$ such that $$(I_H\otimes\rho)\lambda(a) = \sum\limits_{i,j}{h_i\otimes a_{i,j}\otimes h_j}.$$ Set $V_a$ the vector space generated by the $a_{i,j}$'s.

\vu
	
	First, $V_a$ is an $H$-subbicomodule of $B$, i.e., $\rho(V_a)\subseteq V_a\otimes H$ and $\lambda(V_a)\subseteq H\otimes V_a$. Indeed,
	$$
	\begin{array}{rcl}
	\sum\limits_{i,j}{h_i\otimes\rho(a_{i,j})\otimes h_j}
	& = & (I_H\otimes\rho\otimes I_H)(I_H\otimes\rho)\lambda(a)\\
	& = & [I_H\otimes((\rho\otimes I_H)\rho)]\lambda(a)\\
	& = & [I_H\otimes((I_H\otimes\Delta)\rho)]\lambda(a)\\
	& = & (I_H\otimes I_H\otimes\Delta)(I_H\otimes\rho)\lambda(a)\\
	& = & \sum{h_i\otimes a_{i,j}\otimes \Delta(h_j)} \\
	& = & \sum{h_i\otimes a_{i,j}\otimes \sum\limits_k{h'_{j,k}\otimes h_k}} \\
	& = & \sum\limits_{i,j,k}{h_i\otimes a_{i,j}\otimes h'_{j,k} \otimes h_k}.
	\end{array}
	$$
	Hence, for each $i,j$ we have $\rho(a_{i,j})=\sum\limits_k{a_{i,j}\otimes h'_{j,k}}\in V_a\otimes H$, which implies that $V_a$ is a right $H$-subcomodule.
of $B$. 	
	Considering that $\sum\limits_{i,j}{h_i\otimes a_{i,j}\otimes h_j} = (I_H\otimes\rho)\lambda(a) = (\lambda\otimes I_H)\rho(a)$, a similar computation shows that $V_a$ is also a left $H$-subcomodule of $B$ (thus, an $H$-subbicomodule of $B$).

\vd

It follows from the classical theory that $V_a$ is a right and left rational  $H^\ast$-module. Furthermore, $$a = 1_{H^\ast}\triangleright a\triangleleft 1_{H^\ast} =\varepsilon\triangleright a\triangleleft\varepsilon = \sum\limits_{i,j}{a_{i,j}\varepsilon(h_i)\varepsilon(h_j)},$$ hence $a\in V_a$ and $H^\ast\triangleright a\triangleleft H^\ast\subseteq V_a$.
	Conversely, taking the dual basis $\{h_i^\ast\}$ of $H^\ast$, we have
	$$h_m^\ast\triangleright a \triangleleft h_l^\ast = \sum\limits_{i,j} {h_l^\ast(h_i)\otimes a_{i,j}\otimes h_m^\ast(h_j)} = a_{l,m},$$
	so $a_{l,m}\in H^\ast\triangleright a\triangleleft H^\ast$.
	Therefore, $V_a = H^\ast\triangleright a\triangleleft H^\ast$ and $C$ is a right and left rational $H^\ast$-module. This means that $C$ is a $H$-subbicomodule of $B$.

\vu
	
	Now consider $\bar{A}$ the subalgebra of $B$ generated by $C$. Since $\rho$ and $\lambda$ are algebra homomorphisms, then $\bar{A}$ is an $H$-subbicomodule algebra of $B$ containing $A$.

\vu

It remains to show that $\overline{A}$ is the smallest one. Let $M$ be an $H$-subbicomodule algebra of $B$ containing $A$. Thus, $M$ is an $H^\ast$-subbimodule of $B$ containing $A$, and so, for any $a\in A\subseteq M$ we have $f\triangleright a\triangleleft g \in M$ for all $f,g\in H^\ast$ which implies that $M\supseteq \bar{A}$ because $M$ is an algebra.
\end{proof}

\vd

	Let $A$ be a partial $H$-bicomodule algebra. Put $\mathcal{X} = H \otimes A \otimes H$ and consider the maps
	$$
	\begin{array}{rclcrcl}
	\rho:\mathcal{X}&\to&\mathcal{X}\otimes H&\mathrm{ and }&\lambda:\mathcal{X}&\to& H\otimes\mathcal{X}\\
	x&\mapsto& (I_H \otimes I_A\otimes\Delta_H)(x)& &x&\mapsto& (\Delta_H\otimes I_A\otimes I_H)(x).
	\end{array}
	$$

	Clearly, $\mathcal{X}$ is a right and left $H$-comodule algebra via $\rho$ and $\lambda$ respectively, and,  for any $x=h\otimes a\otimes k\in \mathcal{X}$, we have,
	$$
	\begin{array}{rcl}
	(I_H\otimes\rho)\lambda(x)
	& = & (I_H\otimes\rho)(\Delta(h)\otimes a\otimes k)\\
	& = & h_1\otimes\rho(h_2\otimes a\otimes k)\\
	& = & h_1\otimes h_2\otimes a\otimes \Delta(k)\\
	& = & h_1\otimes h_2\otimes a\otimes k_1\otimes k_2\\
	& = & \Delta(h) \otimes a\otimes k_1\otimes k_2\\
	& = & \lambda(h \otimes a\otimes k_1)\otimes k_2\\
	& = & (\lambda\otimes I_H)(h \otimes a\otimes k_1\otimes k_2)\\
	& = & (\lambda\otimes I_H)(h \otimes a\otimes \Delta(k))\\
	& = & (\lambda\otimes I_H)\rho(x).
	\end{array}
	$$
Hence, $(\lambda,\mathcal{X}, \rho)$ is an $H$-bicomodule algebra.

\vu

Furthermore, if the partial $H$-bicomodule algebra structure of $A$ is given by $\rhop$ and $\ldap$, then it is straightforward to check that the map $\theta:A \to \mathcal{X}$ defined by $\theta=(I \otimes  \rhop)\ldap = (\ldap \otimes I) \rhop$ is a monomorphism of algebras.

\vu

As well, it is also straightforward to check that  $(\lambda( \theta(a)) \otimes 1_H)(1_H \otimes \rho(\theta(b))) = (I_H \otimes \theta \otimes I_H )((\ldap( a) \otimes 1_H)(1_H \otimes \rhop(b)))$. Therefore, taking $B$ the subalgebra of $\mathcal{X}$ generated by $\theta(A)$ we have proved the following theorem.

\begin{teo}
	Every partial $H$-bicomodule algebra has a globalization.
\end{teo}
We will refer to the globalization $B$, above constructed, as the \emph{standard globalization} of a given partial $H$-bicomodule algebra $A$.

\vu

\subsection{A particular case of globalization}${}$

\vd

We will suppose in this subsection that $H^0$ separates points. Under this assumption, we will show that the globalization of a partial $H$-bicomodule algebra and the globalization of the corresponding partial $H^0$-bimodule algebra are isomorphic.

\vu

From Proposition \ref{Hbcma=>H0bma}, since ${\mathcal X} = H \otimes A \otimes H$ is an $H$-bicomodule algebra via $I_H\otimes I_A \otimes \Delta_H$ and $\Delta_H \otimes I_A \otimes I_H$, it follows that
 ${\mathcal X}$ is an $H^0$-bimodule algebra by the induced actions
$$
\begin{array}{rcl}
f\triangleright(h\otimes a\otimes k) & = & (h\otimes a\otimes k)^{+0}f((h\otimes a\otimes k)^{+1})\\
& = & (h\otimes a\otimes k_1f(k_2))\\
&&\\
(h\otimes a\otimes k)\triangleleft f & = & f((h\otimes a\otimes k)^{-1})(h\otimes a\otimes k)^{-0}\\
& = & f(h_1)h_2\otimes a\otimes k.
\end{array}
$$

 On the other hand, since $A$ is a partial $H$-bicomodule algebra then $A$ is a partial $H^0$-bimodule algebra by the induced actions. As already shown, $Hom(H^0\otimes H^0, A)$ is an $H^0$-bimodule algebra via
$$
\begin{array}{rl}
f\triangleright \alpha:H^0\otimes H^0&\to A\\
g\otimes k&\mapsto \alpha(g\ast f\otimes k)\\
&\\
\alpha \triangleleft f:H^0\otimes H^0&\to A\\
g\otimes k&\mapsto \alpha(g\otimes f\ast k)
\end{array}
$$
and we have an algebra monomorphism
$$
\begin{array}{rcl}
\varphi:A &\to& Hom(H^0\otimes H^0, A)\\
a & \mapsto & \varphi(a):f\otimes g\mapsto f\Ape a \Apd g.
\end{array}
$$

Now, we define the linear map
$$
\begin{array}{rl}
\Psi:H\otimes A\otimes H&\to Hom(H^0\otimes H^0, A)\\
h\otimes a\otimes k&\mapsto \Psi(h\otimes a\otimes k):f\otimes g\mapsto g(h)af(k).
\end{array}$$

\begin{prop}
	With the above notations, the map $\Psi$ is a monomorphism of $H^0$-bimodule algebras. Moreover, $\Psi\theta = \varphi$.
\end{prop}

\begin{proof}
	The injectivity of $\Psi$ easily follows from the assumption that $H^0$ separates points.

\vu
	
	Let $x=h\otimes a\otimes k$ and $y=h'\otimes a'\otimes k'$ in $\mathcal{X}$, and $f,g \in H^0$. Then,
	$$
	\begin{array}{rcl}
	(\Psi(x)\ast\Psi(y))(f\otimes g)& = & \Psi(x)(f_1\otimes g_1)\Psi(y)(f_2\otimes g_2)\\
	& = & (g_1(h) a f_1(k))(g_2(h') a' f_2(k'))\\
	& = & g_1(h)g_2(h') a a' f_1(k) f_2(k')\\
	& = & g(h h') a a' f(k k')\\
	& = & \Psi(h h'\otimes a a'\otimes k k')(f\otimes g)\\
	& = & \Psi(xy)(f\otimes g)
	\end{array}
	$$
	and
	$$
	\begin{array}{rcl}
	\Psi(1_{H\otimes A\otimes H})(f\otimes g)& = &g(1_H) 1_A f(1_H)\\
	& = & 1_{A} \varepsilon_{H^0}(f) \varepsilon_{H^0} (g)\\
	& = & 1_{A} \varepsilon_{H^0\otimes H^0}(f\otimes g)\\
	& = & 1_{Hom(H^0\otimes H^0, A)}(f\otimes g)
	\end{array}
	$$
	Thus, $\Psi$ is an algebra monomorphism.

\vu
	
	It remains to show that $\Psi$ is an $H^0$-bimodule map. Indeed, for  $a\in A$, $f,g,g'\in H^0$ and $h,k\in H$, we have
	$$
	\begin{array}{rcl}
	[f\triangleright\Psi(h\otimes a\otimes k)](g\otimes g')& = &\Psi(h\otimes a\otimes k)(g\ast f\otimes g')\\
	& = &g'(h)\ a\ (g\ast f)(k)\\
	& = &g'(h)\ a\ g(k_1)f(k_2)\\
	& = &g'(h)\ a\ g(k_1f(k_2))\\
	& = &\Psi[h\otimes a\otimes k_1f(k_2)](g\otimes g')\\
	& = &\Psi[f\triangleright(h\otimes a\otimes k)](g\otimes g')\\
	\end{array}
	$$
	and
	$$
	\begin{array}{rcl}
	{}[\Psi(h\otimes a\otimes k)\triangleleft f](g\otimes g')& = &\Psi(h\otimes a\otimes k)(g\otimes f\ast g')\\
	& = &(f\ast g')(h)\ a\ g(k)\\
	& = &f(h_1) g'(h_2)\ a\ g(k)\\
	& = &g'(f(h_1) h_2)\ a\ g(k)\\
	& = &\Psi[f(h_1) h_2\otimes a\otimes k](g\otimes g')\\
	& = &\Psi[(h\otimes a\otimes k)\triangleleft f](g\otimes g').
	\end{array}
	$$
	
	Therefore, $\Psi$ is an $H^0$-bimodule map.
	
\vu

	Finally, for any $a\in A$,
	$$
	\begin{array}{rcl}
	\Psi(\theta(a))(f\otimes g)& = & \Psi[(\bar\lambda\otimes I_H)\bar\rho(a)](f\otimes g)\\
	& = & \Psi[a^{+0-1}\otimes a^{+0-0}\otimes a^{+1}](f\otimes g)\\
	& = & g(a^{+0-1})\ a^{+0-0}\ f(a^{+1})\\
	& = & (a^{+0}\Apd g)f(a^{+1})\\
	& = & (a^{+0}f(a^{+1}))\Apd g\\
	& = & f\Ape a\Apd g\\
	& = & \varphi(a)(f\otimes g),
	\end{array}
	$$ which prove the last assertion.
\end{proof}

\vu

	As we have seen above, if  $A$ is a partial $H$-bicomodule algebra, then $A$ is also a partial $H^0$-bimodule algebra and, in this context, we have the following
two standard globalizations:
	\begin{itemize}
		\item[(i)] $(H^0\triangleright\varphi(A)\triangleleft H^0, \varphi)$ is a globalization of $A$ as a partial $H^0$-bimodule algebra;	

\vu

		\item[(ii)] $(H^0\triangleright\theta(A)\triangleleft H^0, \theta)$ is a globalization of $A$ as a partial $H$-bicomodule algebra.
	\end{itemize}

\vu

Since $\Psi \theta = \varphi$ and $\Psi$ is a monomorphism of $H^0$-bimodule algebra, it easily follows that

\vu

\begin{prop}
$$\Psi|_{H^0\triangleright\theta(A)\triangleleft H^0}:H^0\triangleright\theta(A)\triangleleft H^0\to H^0\triangleright\varphi(A)\triangleleft H^0$$
is an isomorphism of $H^0$-bimodule algebras.
\end{prop}

\vu

\section{Partial left-right smash products}${}$
\begin{dfn}
	Let $A$ be a partial $H$-bimodule algebra and $\bar{A}$ a partial $H$-bicomodule algebra. The \emph{$(L,R)$-smash product} of $A$ and $\bar{A}$, denoted by $A \natural \bar{A}$,
is defined as the tensor product $A \otimes \bar{A}$ with the multiplication given by $$(a \natural u)(b \natural v) = (a \Apd v\Up)(u\Upm \Ape b) \natural u\Zpm v\Zp,$$ where $a,b \in A$ and $u,v \in \bar{A}$ .
\end{dfn}

\begin{prop}
	Let $(\Ape, A, \Apd)$ be a partial $H$-bimodule algebra and $(\lambda, \bar{A}, \rho)$ a partial $H$-bicomodule algebra.
Then the $(L, R)$-smash product $A\natural\bar{A}$ is an associative algebra.
\end{prop}

\begin{proof} Recall the notations  $$\lambda(u) = u\Zpm \otimes u\Upm \quad \text{and} \quad \rho(u) = u\Zp \otimes u\Up,$$ for all $u\in \bar{A}$.

\vu

Then, for all $a, b, c\in A$ and $u, v, w\in\bar{A}$ we have
	$$
	\begin{array}{rcl}
	\multicolumn{3}{l}{[(a\natural u)(b\natural v)](c\natural w) =} \\
	& = &  [(a\Apd v{\Up})(u{\Upm}\Ape b)\natural u{\Zpm}v{\Zp}](c\natural w)\\
	& = & [(a\Apd v{\Up})(u{\Upm}\Ape b)\Apd w{\Up}][u{\Zpm \Upm}v{\Zp \Upm}\Ape c]\natural u{\Zpm \Zpm}v{\Zp \Zpm}w{\Zp}\\
	& = & (a\Apd v{\Up}\Apd w{\Up_1})(u{\Upm}\Ape b\Apd w{\Upm_2})[u{\Zpm \Upm}v{\Zp \Upm}\Ape c]\natural u{\Zpm \Zpm}v{\Zp \Zpm}w{\Zp}\\
	& = & (a\Apd v{\Up}\Apd w{\Up_1})(u{\Upm_1}\Ape b\Apd w{\Upm_2})[u\Upm_2 {1_A} {\Upm}v{\Zp \Upm}\Ape c]\natural u{\Zpm}{1_A}{\Zpm}v{\Zp \Zpm}w{\Zp}\\
	& = & (a\Apd v{\Up}\Apd w{\Up_1})(u{\Upm_1}\Ape b\Apd w{\Upm_2})[u{\Upm_2}v{\Zp \Upm}\Ape c]\natural u{\Zpm}v{\Zp \Zpm}w{\Zp}\\
	& = & (a\Apd v{\Up} w{\Up_1})(1_A\Apd w{\Upm_2})(u{\Upm_1}\Ape b\Apd w{\Up_3})[u{\Upm_2}v{\Zp \Upm}\Ape c]\natural u{\Zpm}v{\Zp \Zpm}w{\Zp}\\
	& = & (a\Apd v{\Up} w{\Up_1})(u{\Upm_1}\Ape b\Apd w{\Upm_2})[u{\Upm_2}v{\Zp \Upm}\Ape c]\natural u{\Zpm}v{\Zp \Zpm}w{\Zp}\\
	& = &  (a \Apd v{\Up} w{\Up_1})(u{\Upm_1} \Ape b \Apd w{\Upm_2})(u{\Upm_2} v{\Up \Upm} \Ape c) \natural u{\Zpm} v{\Zp \Zpm} w{\Zp} \\
	& = &  (a \Apd v{\Zpm \Up} w{\Up_1})(u{\Upm_1} \Ape b \Apd w{\Upm_2})(u{\Upm_2} v{\Upm} \Ape c) \natural u{\Zpm} v{\Zpm \Zp} w{\Zp} \\
	& = &  (a \Apd v{\Zpm \Up} w{\Up_1})(u{\Upm_1} \Ape b \Apd w{\Upm_2})(u{\Upm_2} \Ape 1_A)(u{\Upm_3} v{\Upm} \Ape c) \natural u{\Zpm} v{\Zpm \Zp} w{\Zp} \\
	& = &  (a \Apd v{\Zpm \Up} w{\Up_1})(u{\Upm_1} \Ape b \Apd w{\Upm_2})(u{\Upm_2} \Ape v{\Upm} \Ape c) \natural u{\Zpm} v{\Zpm \Zp} w{\Zp} \\
	& = &  (a \Apd v{\Zpm \Up} {1_A}{\Up} w{\Up_1})(u{\Upm_1} \Ape b \Apd w{\Upm_2})(u{\Upm_2} \Ape v{\Upm} \Ape c) \natural u{\Zpm} v{\Zpm \Zp} {1_A}{\Zp} w{\Zp} \\
	& = &  (a \Apd v{\Zpm \Up} w{\Zp \Up})(u{\Upm_1} \Ape b \Apd w{\Up})(u{\Upm_2} \Ape v{\Upm} \Ape c) \natural u{\Zpm} v{\Zpm \Zp} w{\Zp \Zp} \\
	& = &  (a \Apd v{\Zpm \Up} w{\Zp \Up})(u{\Upm} \Ape [(b \Apd w{\Up})(v{\Upm} \Ape c)]) \natural u{\Zpm} v{\Zpm \Zp} w{\Zp \Zp} \\
	& = &  (a \natural u)[(b \Apd w{\Up})(v{\Upm} \Ape c) \natural v{\Zpm} w{\Zp}]\\
	& = &  (a \natural u)[(b \natural v)(c \natural w)],
	\end{array}
	$$ and the assertion follows.
\end{proof}

The element $1_A \natural 1_{\bar{A}}$ is not an identity for $A\natural\bar{A}$ and, in general, it is not even an idempotent.

\vu

\begin{ex}
	Let $H = \mathbb{H}_4 = \left\langle 1,g,x,xg\ |\ g^2 = 1, \  x^2 = 0,\ gx=-xg\right\rangle$ be the four dimensional Sweedler's
algebra and $A = \bar{A} = {\Bbbk} $ the ground field.

\vu
	
	Let $\Ape$ be a left partial action of ${H}$ on $\Bbbk$ defined by $$1 \Ape 1_{\Bbbk} = 1_{\Bbbk},\quad g \Ape 1_{\Bbbk} = 0, \quad\text{and}\quad x \Ape 1_{\Bbbk} = r = - xg \Ape 1_{\Bbbk},$$
with $r \in \Bbbk $ . Consider also the right partial action $\Apd$ defined by $$1_{\Bbbk} \Apd 1 = 1_{\Bbbk},\quad 1_{\Bbbk}  \Apd g= 0,\quad\text{and}\quad
1_{\Bbbk}  \Apd x = s = - 1_{\Bbbk}  \Apd xg,$$ with $s\in \Bbbk$.

\vu
	
	Moreover, define a left (resp., right) partial coaction of $H$ in ${\Bbbk}$ by $$\lambda (1_{\Bbbk}) = (\frac{1}{2} + \frac{1}{2}g + txg) \otimes 1_{\Bbbk}\quad
(\text{resp.,}\quad \rho(1_{\Bbbk}) = 1_{\Bbbk} \otimes (\frac{1}{2} + \frac{1}{2}g + ux)),$$  with $t, u \in {\Bbbk}$.

\vu

	Then,
	$$
	\begin{array}{rl}
	(1_{\Bbbk} \natural 1_{\Bbbk})^2 &= (1 \Apd (\frac{1}{2} + \frac{1}{2}g + ux))( (\frac{1}{2} + \frac{1}{2}g + txg) \Ape 1_{\Bbbk}) \natural 1_{\Bbbk} \\
	&= (\frac{1}{2} + ts)(\frac{1}{2} - ur) \natural 1_{\Bbbk}
	\end{array}
	$$
	
	And, choosing $ts-\frac{1}{2}$ or $ur=\frac{1}{2}$ we obtain $(1_{\Bbbk} \natural 1_{\Bbbk})^2 = 0$.
\end{ex}

\vu

We end this section discussing conditions under which one can obtain  a unital algebra from  $A\natural\bar{A}$. Indeed, it is enough to look for conditions in
order to ensure the existence of a nonzero idempotent in $A\natural\bar{A}$.  We start by the following proposition (here, $Ker(\varepsilon)$ denotes the kernel of $\varepsilon$).

\begin{prop}\label{prop4}
	\label{pp1}
	Let $(\Ape, A)$ be a left partial $H$-module algebra and a nonzero element $a \in A$. Then,  $Ker(\varepsilon) \Ape a = 0$ if and only if $h \Ape a = \varepsilon(h) a$, for all $h \in H$.
\end{prop}

\begin{proof}
	Note that $H = Ker(\varepsilon) \oplus {\Bbbk} 1_H$. Taking $h \in H$, $h = h_{\varepsilon} + \alpha 1_H$, with $h_{\varepsilon} \in Ker(\varepsilon)$ and $\alpha \in {\Bbbk}$,
	we have that $\varepsilon(h) = \varepsilon(h_{\varepsilon}) + \alpha \varepsilon(1_H) = \alpha$, and hence
$h \Ape a = h_{\varepsilon} \Ape a + \alpha (1_H \Ape a) = \alpha a=\varepsilon(h)a$.
	Conversely, if $h \in Ker(\varepsilon)$, then $h \Ape a = \varepsilon(h) a = 0$.
\end{proof}

If $(\Ape, A, \Apd)$ is a partial $H$-bimodule algebra we can extend the above result to the following statement
$$Ker(\varepsilon) \Ape a = 0 = a \Apd Ker(\varepsilon)\, \ \text{if and only if}\, \  h \Ape a = \varepsilon(h) a = a \Apd h,\,\ \text{for all}\,\ h \in H.$$

\vu

\begin{prop}\label{prop5}
	Let $(\lambda, \bar{A}, \rho)$ be a partial $H$-bicomodule algebra, $(\Ape,A, \Apd)$ a partial $H$-bimodule algebra. Assume that there exist nonzero elements $a \in A$ such that $a^2 =a$ and
$Ker(\varepsilon) \Ape a = 0 = a \Apd Ker(\varepsilon)$, and $u \in \bar{A}$ such that $u^2 = u$. Then, $(a \natural u)^2 = a \natural u$ in $A \natural \bar{A}$.
\end{prop}

\begin{proof}
	$$\begin{array}{rcl}
	(a \natural u)^2
	& = & (a \Apd u^{+1})(u ^{-1} \Ape a) \natural u ^{-0} u ^{+0}\\
	& \stackrel{(\ref{pp1})}{=} & a^2 \natural u ^{-0} \varepsilon(u ^{-1}) u ^{+0} \varepsilon(u ^{+1})\\
	& = & a \natural u^2\\
	& = & a \natural u.
	\end{array}
	$$
\end{proof}

\begin{prop}
	Let $A$ and $\bar{A}$ be as in Proposition \ref{prop5}. Assume that there exist nonzero idempotents $a\in A$ and $u\in \bar{A}$.
  Consider the following conditions:
	\begin{enumerate}\Not{}
		\item $h \Ape a = \varepsilon(h) a$, for all $h \in H$,
		\item $a \Apd h = \varepsilon(h) a$, for all $h \in H$,
		\item $\rho(u) = u \otimes 1_H$,
		\item $\lambda(u) = 1_H \otimes u$.
	\end{enumerate}
	
	In each one of the cases (1)+(2), (1)+(4), (2)+(3) or (3)+(4),  one always has  $(a \natural u)^2 = (a \natural u)$.
\end{prop}

\begin{proof}
	It is straightforward.
\end{proof}
If, in Proposition \ref{prop5}, the involved left and right (co)actions are all global, then all the conditions  (1)-(4) are satisfied.

\vu

The proof of the next proposition is immediate.

\begin{prop}
	Let $A$ be a partial $H$-bimodule algebra and $\bar{A}$ a partial $H$-bicomodule algebra such that
$(a \natural u)^2 = a \natural u$ for some $a \in A$ and $u \in \bar{A}$.
Then $A \underline{\natural} \bar{A} = \{(a \natural u)(x \natural y)(a \natural u)| \ x \in A,  y \in \bar{A}\}$
is a unital algebra with identity $a \natural u$.\qed
\end{prop}

Note that this last proposition gives a generalization of the smash product as well as of the global $(L,R)$-smash product.

\bibliographystyle{amsalpha}

\end{document}